\documentclass[a4paper,11pt,reqno,oneside,abstracton]{scrartcl}
\usepackage{scalerel,stackengine,relsize}
\stackMath
\usepackage{amssymb}
\usepackage{bbm}
\usepackage{ifthen}
\usepackage{tikz}
\usepackage{tikz-3dplot}
\usepackage{bm}
\usepackage[hyphens]{url} 
\usepackage{amsmath,amssymb,amsfonts,amsthm,amsopn,dsfont,mathrsfs}
\usepackage{esint}
\usepackage{graphicx}
\usepackage{enumitem}
\usepackage{bigints}
\usepackage{fancyhdr}
\usepackage{hyperref}
\hypersetup{
    colorlinks=true,
    linkcolor=blue,
    filecolor=magenta,      
    urlcolor=cyan,
    }
\usepackage[T1]{fontenc}
\usepackage{mathtools}
\usepackage[utf8]{inputenc}
\usepackage{verbatim} 
\usepackage[backend=bibtex,style=numeric,bibencoding=utf8,language=auto,autolang=other,maxnames=10]{biblatex}
\usepackage[a4paper, total={6in, 9in}]{geometry}

\newcommand*\quot[2]{{^{\textstyle #1}\big/_{\textstyle #2}}}
\newcommand\blfootnote[1]{%
  \begingroup
  \renewcommand\thefootnote{}\footnote{#1}%
  \addtocounter{footnote}{-1}%
  \endgroup
}
\newtheorem{thm}{Theorem}[section]
\newtheorem{defn}{Definition}[section]
\newtheorem{lem}[thm]{Lemma}
\newtheorem{prop}[thm]{Proposition}
\newtheorem{cor}[thm]{Corollary}
\newtheorem{rem}{Remark}

\title{The fractional Hopf differential and a weak formulation of stationarity for the half Dirichlet energy}
\author{Filippo Gaia}
\date{\today}

\begin{document}

\allowdisplaybreaks

\maketitle
\begin{abstract}
We obtain a weak formulation of the stationarity condition for the half Dirichlet energy, which can be expressed in terms of a fractional analogous to the Hopf differential. As an application we show that conformal harmonic maps from the disc are precisely the harmonic extensions of stationary points of the half Dirichlet energy on the circle. We also derive a Noether theorem and a Pohozaev identity for stationary points of the half Dirichlet energy.
\end{abstract}

\blfootnote{\textup{2020} \textit{Mathematics Subject Classification}: \textup{58E20, 35R11, 35J20}}

\section{Introduction}
In the study of harmonic maps from a domain $\Omega\subset\mathbb{R}^n$ to an embedded manifold $N\subset\mathbb{R}^k$, the stationarity condition
\begin{align}\label{eq: stationarity condition}
    \frac{d}{dt}\bigg\vert_{t=0}\int_\Omega\lvert \nabla u (x+tX(x))\rvert^2=0
\end{align}
(for any compactly supported smooth vector field $X$ on $\Omega$) plays a crucial role. In dimension 3 it is a necessary condition for any partial regularity result (see \cite{Riviere-Acta}), while when $n=2$ (which is critical for the harmonic maps problem), condition \eqref{eq: stationarity condition} is satisfied by all weakly harmonic maps as a consequence of their regularity (see \cite{Helein-regularity}).
For any map $u$ in the natural space $W^{1,2}(\Omega,N)$, condition \eqref{eq: stationarity condition} is equivalent to the following Euler-Lagrange equation:
\begin{align}\label{eq: EL for stationarity}
    \sum_{i=1}^n \partial_{x_i}\left(\lvert\nabla u\rvert^2\delta_{ij}-2\partial_{x_i}u\cdot\partial_{x_j}u\right)\quad\forall\, j=1,...,n.
\end{align}
Notice that for a general map $u\in W^{1,2}(\Omega, N)$, \eqref{eq: EL for stationarity} only makes sense in a distributional sense. In the special case $n=2$, the system of equations \eqref{eq: EL for stationarity} can be rewritten in complex coordinates as
\begin{align}\label{eq: HD is holomorphic}
    \partial_{\overline{z}}\mathscr{H}(u)=0,
\end{align}
where $\mathscr{H}(u)$ is the Hopf differential of $u$, defined as
\begin{equation}\label{eq: Hopf-Diff}
\mathscr{H}(u)=\frac{\partial u}{\partial z}\cdot\frac{\partial u}{\partial z}.
\end{equation}
Recall that $\mathscr{H}(u)=0$ if and only if $u$ is weakly conformal. As conformal harmonic maps are parametrizations of minimal surfaces (in $N$), condition \eqref{eq: HD is holomorphic} is of particular interest in geometric analysis. Consider for instance the following problem: let $N=\mathbb{R}^k$, $\Omega=D^2$ and let $M\subset\mathbb{R}^k$ be a submanifold. 
Let $u_0\in H^\frac{1}{2}(\partial D^2)$,
\begin{enumerate}
    \item when is its harmonic extension $u$ (i.e. the only harmonic map from $D^2$ to $\mathbb{R}^k$ having $u_0$ as a trace) conformal?
    \item when does it define a free boundary minimal surface (with boundary in $M$)?
\end{enumerate}
For the second question, recall that $u$ satisfies the free boundary condition on $\partial D^2$ if and only if $P_M^T (-\Delta)^\frac{1}{2}u_0=0$ (where $P_M^T$ denotes the orthogonal projection on the tangent space of $M$), i.e. $u_0$ is half harmonic. In this case, $u_0$ is smooth (see \cite{3-terms-commutators}) and the holomorphicity of the Hopf differential of $u$ \eqref{eq: HD is holomorphic} implies that $u$ is conformal (see Theorem 1.2 in \cite{DALIO-compactness}, for a proof see Lemma 4.27 in \cite{Millot-Sire} or Theorem 2.9 in \cite{DaLio-Martinazzi-Riviere}; see also Theorem 1.6 in \cite{DaLio-Pigati}). Regarding the first question, we will see later that with the help of condition \eqref{eq: HD is holomorphic} we can characterize the maps $u_0\in H^\frac{1}{2}(\partial D^2, M)$ whose harmonic extension is conformal (see Theorem \ref{thm: Thm conformal extension, intro}).\\
The purpose of the present work is to study the stationarity condition in the analogous\footnote{Notice that dimension 1 is critical for the half harmonic maps problem (see Proposition 1 in \cite{Sub-criticality}). Moreover, while the Dirichlet energy is invariant under conformal transformations in dimension 2, the half Dirichlet energy on $\mathbb{S}^1$ is invariant under the traces of conformal transformations (see Lemma \ref{lem: conformal-invariance}, the corresponding invariance property for the half Dirichlet energy on $\mathbb{R}$ was already observed in \cite{3-terms-commutators}).} framework of half harmonic maps on $\mathbb{S}^1$ and present some applications to local and non-local problems.
First we obtain a Euler-Lagrange equation for stationary points of the half Dirichlet energy\footnote{For any $s\in (0,1)$, $(-\Delta)^s u$ denotes the $s$-fractional Laplacian of $u$, defined through the following identity for the Fourier coefficients:
\begin{align}
    \widehat{(-\Delta)^s u}(n)=\lvert n\rvert^{2s}\widehat{u}(n)\qquad \forall \, n\in \mathbb{Z}.
\end{align}}
\begin{align}\label{eq: def-s-frac-energy}
    E_\frac{1}{2}(u)=\int_{\mathbb{S}^1}\left\lvert (-\Delta)^\frac{1}{4} u\right\rvert^2.
\end{align}
A direct computation shows that when $u$ is sufficiently regular, $u:\partial D^2\to \mathbb{R}^k$ is a stationary point of $E_\frac{1}{2}$ if and only if
\begin{align}\label{eq: weak EL}
    (-\Delta)^\frac{1}{2}u\cdot u'=0.
\end{align}
This expression however might not be well defined for some $u$ in the natural space $H^\frac{1}{2}(\partial D^2)$. We show in Proposition \ref{prop: first-variation-formula} that \eqref{eq: weak EL} can be given a distributional meaning for any $u\in H^\frac{1}{2}(\partial D^2)$, providing a weak formulation of the stationarity condition for the half Dirichlet energy, in analogy to \eqref{eq: EL for stationarity}.
We will then see that stationary points of the half Dirichlet energy can be characterized in terms of the trace of the Hopf differential of their harmonic extensions. More precisely, notice first that even if we assume $v\in H^1\cap C^\infty(D^2)$, the trace of $\mathscr{H}(v)$ on $\partial D^2$ might not be well defined. Consider the following operator, extending the standard trace operator for functions smooth up to the boundary: for any $u\in H^\frac{1}{2}(\mathbb{S}^1)$ let $\tilde{u}$ denote its harmonic extension in $D^2$. For any $\varphi\in C^\infty(\mathbb{S}^1)$ set
\begin{align}
    \mathscr{H}_\frac{1}{2}(u)[\varphi]=\int_{D^2}\left(\frac{1}{z}\mathscr{H}(\tilde{u})\frac{\partial\tilde{\varphi}}{\partial \overline{z}}\right) dx^2,
\end{align}
where $\tilde{\varphi}$ is the harmonic extension of $\varphi$ in $D^2$.
We will call the distribution $\mathscr{H}_\frac{1}{2}(u)$ the \textit{fractional Hopf differential of $u$ }and we will see that it is related to the first inner variation of $E_\frac{1}{2}$ by the following formula (see Proposition \ref{prop: frac-Hopf-diff}):
\begin{align}\label{eq: frac-Hopf-diff-variation-statement}
    \mathscr{H}_\frac{1}{2}(u)=\frac{e^{-i2\theta}}{2i}\left(\mathcal{V}_\frac{1}{2}(u)+iH\left(\mathcal{V}_\frac{1}{2}(u)\right)\right),
\end{align}
where $H$ denotes the Hilbert's transform.
In particular we have
\begin{lem}\label{lem: characterization FHD}
    $u\in H^\frac{1}{2}(\partial D^2)$ is a stationary point of the half Dirichlet energy if and only if $\mathscr{H}_\frac{1}{2}(u)=0$.
\end{lem}
This will allow us to show that harmonic extensions of stationary points of $E_\frac{1}{2}$ are exactly the harmonic conformal maps from $D^2$:
\begin{thm}\label{thm: Thm conformal extension, intro}
    Let $u\in H^\frac{1}{2}(\mathbb{S}^1)$ and let $\tilde{u}\in H^1(D^2)$ denote its harmonic extension. Then the following are equivalent.
    \begin{enumerate}
        \item $u$ is a stationary point of $E_\frac{1}{2}$, i.e. $\mathscr{H}_\frac{1}{2}(u)=0$,
        \item $\tilde{u}$ is weakly conformal, i.e. $\mathscr{H}(\tilde{u})=0$.
    \end{enumerate}
\end{thm}
From Lemma \eqref{lem: characterization FHD} we also deduce a characterization of stationary points of the half Dirichlet energy in terms of Fourier coefficients (which has already been known from \cite{dalio2016pohozaevtype}, \cite{DaLio-Pohozaev}, see also \cite{Cabre-Mas}):
\begin{lem}
    Let $u\in H^\frac{1}{2}(\mathbb{S}^1)$. Then $u$ is a stationary point of $E_\frac{1}{2}$ if and only if for any $k\in \mathbb{N}$ there holds
    \begin{align}\label{eq: proof-characterization-stationarity-Fourier-intro}
        \sum_{\substack{m,n\in\mathbb{N}\\ m+n=k}}m n\,\widehat{u}(m)\cdot \widehat{u}(n)=0.
    \end{align}
    In particular for any stationary point $u\in H^\frac{1}{2}(\mathbb{S}^1)$ of $E_\frac{1}{2}$ there holds
    \begin{align}\label{eq: balancing condition}
        \left\lvert\int_0^{2\pi}u(\theta)\cos(\theta)d\theta\right\rvert^2=\left\lvert\int_0^{2\pi}u(\theta)\sin(\theta)d\theta\right\rvert^2
    \end{align}
\end{lem}
Notice that \eqref{eq: proof-characterization-stationarity-Fourier-intro}, \eqref{eq: balancing condition} are true in particular for half harmonic maps (which are stationary as they are smooth, see \cite{Sub-criticality}). We remark that \eqref{eq: balancing condition} can be interpreted as a balancing condition analogous to the Pohozaev identity
\begin{align}
    \int_{\partial B_r}\left\lvert\frac{1}{r}\partial_\theta u\right\rvert^2=\int_{\partial B_r}\left\lvert\partial_r u\right\rvert^2
\end{align}
for harmonic functions in two dimensional domains,
which follows from the stationarity under dilations in the domain. Equation \eqref{eq: balancing condition} plays an important role in the study of the compactness properties of horizontal half harmonic maps, in particular for showing that there is no dissipation of energy in the "neck regions" along a bubbling sequence (see \cite{dalio2016pohozaevtype}, see also \cite{Laurain-Riviere} for the analogous argument for harmonic maps in dimension 2). For other fractional counterparts of the Pohozaev identity see \cite{RosOton-Serra} (see also Section 5 in \cite{gaia-MT}).\\
Finally, the weak formulation of the stationarity condition for $E_\frac{1}{2}$ will allow us to exploit the symmetry properties of the half Dirichlet energy to deduce Noether theorems for stationary points (Theorem \ref{prop_final_version_Noether_thm_vardom_12}), in analogy to \cite{Helein-book}, Theorem 1.3.6 (see also the discussion thereafter).

\textbf{Acknowledgements:} the present work is based on some chapters of the author's Master thesis \cite{gaia-MT}. The author would like to thank prof. Tristan Rivière, prof. Xavier Ros-Oton and Alessandro Pigati for their guidance and their support.

\section{The first inner variation}

Assume that $u\in C^\infty(\mathbb{S}^1)$, then 
\begin{align}
    \frac{d}{dt}\bigg\vert_{t=0}E_\frac{1}{2}(u\circ\varphi_t)
    =&\int_{\mathbb{S}^1}\frac{d}{dt}\bigg\vert_{t=0}\left\lvert(-\Delta)^\frac{1}{4}u\circ\varphi_t\right\rvert^2
    =\int_{\mathbb{S}^1}(-\Delta)^\frac{1}{4}u\cdot(-\Delta)^\frac{1}{4}(u'X)\\
    \nonumber
    =&\int_{\mathbb{S}^1}(-\Delta)^{\frac{1}{2}}u\cdot u'X.
\end{align}
For any such $u$ we set
\begin{align}
    \mathcal{V}_\frac{1}{2}(u):=(-\Delta)^{\frac{1}{2}}u\cdot u'.
\end{align}
While the energy $E_\frac{1}{2}(u)$ is defined naturally for any $u\in H^\frac{1}{2}(\mathbb{S}^1)$, it is not immediately clear how to generalize the definition of $\mathcal{V}_\frac{1}{2}(u)$ to a general $u\in H^\frac{1}{2}(\mathbb{S}^1)$.
The following results will allow us to extend the definition of $\mathcal{V}_\frac{1}{2}(u)$ by continuity, in a distributional sense.\\
\begin{defn}
    Let $s\in \left(0,\frac{1}{2}\right]$. For any $a,b\in C^\infty(\mathbb{S}^1)$ set
    \begin{align}
        D_s(a,b)=(-\Delta)^sa\, b-(-\Delta)^sb\,a.
    \end{align}
\end{defn}
The algebraic structure of $D_s(a,b)$ will allow us to derive non-trivial estimates for it in terms of Sobolev norms of $a$ and $b$ (see Lemma \ref{lem: estimate for div} and Lemma \ref{lem: estimate div_1/commutator}), which permit to extend by continuity the meaning of $D_s(a,b)$ in a distributional sense, even when $a,b$ have low regularity. Observe that the operator $D_s$ is related to the fractional divergence introduced in \cite{Mazowiecka-Schikorra} (see Lemma \ref{lem: link-fractional-divergence}).
\begin{lem} For any $u,\varphi\in C^\infty(\mathbb{S}^1)$ we have 
\begin{align}
    \left\lvert \int_{\mathbb{S}^1}(-\Delta)^s u\cdot u'\varphi\right\rvert\leq C\lVert \varphi'\rVert_{\mathbb{A}}\lVert u\rVert_{H^s}^2
\end{align}
for some independent constant $C$.
\end{lem}
For the definition of the space $\mathbb{A}$, see Definition \ref{defn: Wiener-algebra}.
\begin{proof}
Notice that we can rewrite $(-\Delta)^\frac{1}{2}u\cdot u'$ as
\begin{align}
    u'\cdot(-\Delta)^su=&\frac{1}{2}\frac{d}{dx}\lvert(-\Delta)^\frac{s}{2}u\rvert^2-(-\Delta)^\frac{s}{2}u'\cdot(-\Delta)^\frac{s}{2}u+u'\cdot (-\Delta)^su\\
    \nonumber
    =&\frac{1}{2}\frac{d}{dx}\lvert(-\Delta)^\frac{s}{2}u\rvert^2+D_\frac{s}{2}\left((-\Delta)^\frac{s}{2}u, u'\right).
\end{align}
The statement now follows form Lemma \ref{lem: estimate for div}.
\end{proof}

\begin{lem}\label{lem: continuity of the derivative}
    Let $X$ be a smooth vector field on $\mathbb{S}^1$ and let $\phi_t$ denote its flow. Then the function
    \begin{align}
        t\mapsto E_\frac{1}{2}(u\circ\varphi_t)
    \end{align}
    is of class $C^1$ and the map
    \begin{align}
        \mathscr{N}: H^\frac{1}{2}(\mathbb{S}^1)\to \mathbb{R},\quad u\mapsto \frac{d}{dt}\bigg\vert_{t=0}E_\frac{1}{2}(u\circ\phi_t)
    \end{align}
    is continuous.
\end{lem}
\begin{proof}
    First we recall that for any $u\in H^{\frac{1}{2}}(\mathbb{S}^1)$, the energy $E_\frac{1}{2}$ can be rewritten as
    \begin{align}
        E_\frac{1}{2}(u)=\frac{1}{2\pi}\int_{\mathbb{S}^1}\int_{\mathbb{S}^1}\frac{\lvert u(x)-u(y)\rvert^2}{\lvert x-y\rvert^{2}}dxdy
    \end{align}
    (see Lemma B 8 in \cite{DaLio-Pigati}). Let $\{u_n\}_{n\in \mathbb{N}}$ be a sequence of smooth functions approximating $u$ in $H^s(\mathbb{S}^1)$.
    For any $t\in \mathbb{R}$ a change of variables yields
    \begin{align}
        E_\frac{1}{2}(u\circ\phi_t)=&\frac{1}{2\pi}\int_{\mathbb{S}^1}\int_{\mathbb{S}^1}\frac{\lvert u(\phi_t(x))-u(\phi_t(y))\rvert^2}{\lvert x-y\rvert^{2}}dxdy\\
        \nonumber
        =&\frac{1}{2\pi}\int_{\mathbb{S}^1}\int_{\mathbb{S}^1}\frac{\lvert u(x')-u(y')\rvert^2}{\lvert \phi_{-t}(x')-\phi_{-t}(y')\rvert^{2}}\lvert \phi_{-t}'(x')\rvert \lvert \phi_{-t}'(y')\rvert dx'dy'.
    \end{align}
    Notice that when $t$ is sufficiently small we have
    \begin{align}
        \left\lvert \lvert \phi_{-t}'(x)\rvert \lvert \phi_{-t}'(y)\rvert-1\right\rvert\leq C\lvert t\rvert,\qquad \frac{1}{\lvert \phi_{-t}(x)-\phi_{-t}(y)\rvert^{2}}\leq \frac{C}{\lvert x-y\rvert^2}
    \end{align}
    and
    \begin{align}
        \left\lvert\frac{1}{\lvert \phi_{-t}(x)-\phi_{-t}(y)\rvert^{2}}-\frac{1}{\lvert x-y\rvert^{2}}\right\rvert\leq C\frac{\lvert t\rvert}{\lvert x-y\rvert^{2}}
    \end{align}
    for any $x,y\in \mathbb{S}^1$, for a constant $C$ independent form $x,y,t$.
    Therefore 
    \begin{align}\label{eq: estimate-for-differentiability}
        &\left\lvert E_s(u_n\circ\phi_t)-E_s(u_n)-(E_s(u\circ\phi_t)-E_s(u))\right\rvert=\\
        \nonumber
        &\left\lvert\frac{1}{2\pi}\int_{\mathbb{S}^1}\int_{\mathbb{S}^1}\frac{\lvert u_n(x)-u_n(y)\rvert^2-\lvert u(x)-u(y)\rvert^2}{\lvert \phi_{-t}(x)-\phi_{-t}(y)\rvert^{2}}\lvert \phi_{-t}'(x)\rvert \lvert \phi_{-t}'(y)\rvert\right.\\
        \nonumber
        &\left.-\frac{1}{2\pi}\int_{\mathbb{S}^1}\int_{\mathbb{S}^1}\frac{\lvert u_n(x)-u_n(y)\rvert^2-\lvert u(x)-u(y)\rvert^2}{\lvert x-y\rvert^{2}}dxdy\right\rvert\\
        \nonumber
        &\leq\frac{1}{2\pi}\int_{\mathbb{S}^1}\int_{\mathbb{S}^1}\left\lvert\frac{\lvert u_n(x)-u_n(y)\rvert^2-\lvert u(x)-u(y)\rvert^2}{\lvert \phi_{-t}(x)-\phi_{-t}(y)\rvert^{2}}\right\rvert\left\lvert\lvert \phi_{-t}'(x)\rvert \lvert \phi_{-t}'(y)\rvert-1\right\rvert dxdy\\
        \nonumber
        &+\frac{1}{2} \int_{\mathbb{S}^1}\int_{\mathbb{S}^1}\left\lvert\lvert u_n(x)-u_n(y)\rvert^2-\lvert u(x)-u(y)\rvert^2\right\rvert\left\lvert\frac{1}{\lvert \phi_{-t}(x)-\phi_{-t}(y)\rvert^{2}}-\frac{1}{\lvert x-y\rvert^{2}}\right\rvert dxdy\\
        \nonumber
        &\leq C\lvert t\rvert \int_{\mathbb{S}^1}\int_{\mathbb{S}^1}\left\lvert\frac{\lvert u_n(x)-u_n(y)\rvert^2-\lvert u(x)-u(y)\rvert^2}{\lvert x-y\rvert^{2}}\right\rvert dxdy.
    \end{align}
We claim that the integral in the last line of \eqref{eq: estimate-for-differentiability} converges to $0$ as $n\to \infty$.
To this end we define
\begin{equation}
f: \mathbb{S}^1\times \mathbb{S}^1\to \mathbb{R},\quad(x,y)\mapsto \frac{\lvert u(x)-u(y)\rvert^2}{\left\lvert x-y\right\rvert^{2}}
\end{equation}
and for any $n\in\mathbb{N}$
\begin{equation}
f_n: \mathbb{S}^1\times \mathbb{S}^1\to \mathbb{R},\quad(x,y)\mapsto \frac{\lvert u_n(x)-u_n(y)\rvert^2}{\left\lvert x-y\right\rvert^{2}}.
\end{equation}
To prove the claim, observe that for any $n\in\mathbb{N}$
\begin{equation}
\lvert f\rvert+\lvert f_n\rvert-\lvert f_n-f\rvert\geq 0.
\end{equation}
Since $f_n\to f$ a.e., Fatou's Lemma implies
\begin{equation}
\int_{\mathbb{S}^1}\int_{\mathbb{S}^1} 2\lvert f\rvert\leq\liminf_{n\to\infty}\int_{\mathbb{S}^1}\int_{\mathbb{S}^1} \lvert f\rvert+\lvert f_n\rvert-\lvert f_n-f\rvert.
\end{equation}
Since $u_n\to u$ in $H^\frac{1}{2}(\mathbb{S}^1)$ , $\lVert f_n\rVert_{L^1}\to\lVert f\rVert_{L^1}$ as $n\to \infty$, and thus we conclude that
\begin{equation}
\limsup_{n\to\infty}\int_{\mathbb{S}^1}\int_{\mathbb{S}^1}\lvert f_n-f\rvert=0.
\end{equation}
This concludes the proof of the claim.
Now we observe that dividing the first and the last line of \eqref{eq: estimate-for-differentiability} by $t\neq 0$ and letting $t$ tend to $0$, we obtain
\begin{equation}
\begin{split}
&\limsup_{n\to \infty} \limsup_{t\to 0} \left( \frac{E_\frac{1}{2}(u_n\circ  \phi_t)-E_\frac{1}{2}(u_n)}{t}-\frac{E_\frac{1}{2}(u\circ \phi_t)-E_\frac{1}{2}(u)}{t}\right)\leq 0,\\
&\liminf_{n\to \infty} \liminf_{t\to 0} \left( \frac{E_\frac{1}{2}(u_n\circ  \phi_t)-E_\frac{1}{2}(u_n)}{t}-\frac{E_\frac{1}{2}(u\circ \phi_t)-E_\frac{1}{2}(u)}{t}\right)\geq 0.
\end{split}
\end{equation}
For any $n\in \mathbb{N}$, since $u_n$ is smooth, $E(u_n\circ\phi_t)$ is differentiable, therefore
\begin{equation}
\begin{split}
&\limsup_{n\to \infty} \frac{d}{dt}\bigg\vert_{t=0}E_\frac{1}{2}(u_n\circ\phi_t)+\limsup_{t\to 0}-\frac{E_\frac{1}{2}(u\circ \phi_t)-E_\frac{1}{2}(u)}{t}\leq 0,\\
&\liminf_{n\to \infty} \frac{d}{dt}\bigg\vert_{t=0}E_\frac{1}{2}(u_n\circ\phi_t)+\liminf_{t\to 0}-\frac{E_\frac{1}{2}(u\circ \phi_t)-E_\frac{1}{2}(u)}{t}\geq 0.
\end{split}
\end{equation}
Therefore $E(u\circ\phi_t)$ is differentiable in $t$ in $0$ (and thus at any $t$), and
\begin{equation}\label{eq_lem_continuity_derivative_energy_2}
\frac{d}{dt}\bigg\vert_{t=0} E_\frac{1}{2}(u_n\circ \phi_t)\to \frac{d}{dt}\bigg\vert_{t=0} E_\frac{1}{2}(u\circ \phi_t)
\end{equation}
as $n\to \infty$. This implies that the map $\mathscr{N}$ is continuous.
\end{proof}
We can then extend the definition of the first inner variation of $E_\frac{1}{2}$ as follows: let $u\in H^\frac{1}{2}(\mathbb{S}^1)$, let $\{u_n\}_{n\in \mathbb{N}}$ be a sequence of smooth maps converging to $u$ in $H^\frac{1}{2}(\mathbb{S}^1)$. Then
\begin{align}\label{eq: definition of Vs}
    \mathcal{V}_\frac{1}{2}(u):=\lim_{n\to\infty}(-\Delta)^\frac{1}{2}u_n u_n',
\end{align}
where the limit is taken in the space of distributions on $\mathbb{S}^1$. Notice that by Lemma \ref{lem: continuity of the derivative} the definition of $\mathcal{V}_s$ does not depend on the choice of the approximating sequence.
We call the distribution $\mathcal{V}_\frac{1}{2}(u)$ \textbf{the first inner variation of $E_\frac{1}{2}$ at $u$}
For any smooth vector field $X$ with flow $\phi_t$ we have
\begin{align}
    \langle\mathcal{V}_\frac{1}{2}(u), X\rangle
    =&\lim_{n\to\infty}\int_{\mathbb{S}^1}(-\Delta)^\frac{1}{2} u_n u_n' X=\lim_{n\to\infty}\frac{d}{dt}\bigg\vert_{t=0}E_\frac{1}{2}(u_n\circ\phi_t)\\
    \nonumber
    =&\frac{d}{dt}\bigg\vert_{t=0}E_\frac{1}{2}(u\circ\phi_t).
\end{align}
Recall that a map $u\in H^\frac{1}{2}(\mathbb{S}^1)$ is called \textbf{a stationary point of $E_\frac{1}{2}$} if $\frac{d}{dt}\vert_{t=0}E_\frac{1}{2}(u\circ\phi_t)=0$. The above computation shows that this is equivalent to $\mathcal{V}_\frac{1}{2}(u)=0$.\\
We summarize the previous discussion in the following proposition.
\begin{prop}\label{prop: first-variation-formula}
    Let $u\in H^\frac{1}{2}(\mathbb{S}^1)$. Let $X$ be a smooth vector field on $\mathbb{S}^1$ and let $\phi_t$ denote its flow. Then 
    \begin{align}
    \langle\mathcal{V}_\frac{1}{2}(u), X\rangle=\frac{d}{dt}\bigg\vert_{t=0}E_\frac{1}{2}(u\circ\phi_t),
\end{align}
    where $\mathcal{V}_\frac{1}{2}(u)$ is the distribution defined in \eqref{eq: definition of Vs}.
    In particular, $u$ is a stationary point of $E_\frac{1}{2}$ if and only if $\mathcal{V}_\frac{1}{2}(u)=0$.\\
    The map
    \begin{align}
        \mathcal{V}_\frac{1}{2}: H^\frac{1}{2}(\mathbb{S}^1)\to \mathscr{D}'(\mathbb{S}^1),\quad u\mapsto \mathcal{V}_\frac{1}{2}(u)
    \end{align}
    is continuous.
\end{prop}

\section{The fractional Hopf differential}
In this section we investigate how stationary points of $E_\frac{1}{2}$ are related to stationary points of the Dirichlet energy
\begin{align}
    E(v)=\frac{1}{2}\int_{D^2}\lvert\nabla v\rvert^2
\end{align}
(with values in an Euclidean space $\mathbb{R}^k$). Recall that a function $v\in H^1(D^2)$ is a stationary point of $E$ (for compactly supported variations in the domain) if and only if it its Hopf differential
\begin{equation}\label{eq_definition_Hopf_differential_D2}
\mathscr{H}(v)=\frac{\partial v}{\partial z}\cdot \frac{\partial v}{\partial z}
\end{equation}
is holomorphic (see Lemma 1.1 in \cite{Schoen-Analtic-aspects}).
Notice that in polar coordinates, the Hopf differential of a function $v\in H^1(D^2)$ is given by
\begin{equation}\label{eq_comm_Hopf_diff_in_polar_coordinates}
\mathscr{H}(v)(z)=\frac{\overline{z}^2}{4 r^2}\left[\lvert \partial_rv(z)\rvert^2-\frac{1}{r^2}\lvert\partial_\theta v(z)\rvert^2-2i\partial_r v(z)\cdot\frac{1}{r}\partial_\theta v(z)\right]
\end{equation}
for any $z\in D^2$, where $z=re^{i\theta}$.
Notice also that $\mathscr{H}(v)=0$ if and only if $v$ is weakly conformal.\\
If $v\in H^1(D^2)$, $\mathscr{H}(v)$ is a function in $L^1(D^2)$. If $v$ is harmonic, however, it is possible to define the trace of $\mathscr{H}(v)$ as follows.
First notice that if $v$ is harmonic then it is a stationary point of $E$ and thus
\begin{align}
    \frac{\partial}{\partial {\overline{z}}} \mathscr{H}(v)=0.
\end{align}
If $v$ is smooth up to the boundary of $D^2$,
for any complex valued $\varphi\in C^\infty(\overline{D^2})$ there holds
\begin{align}\label{eq: computation-trace-Hopf}
    \int_{\partial D^2}\mathscr{H}(v)\varphi d\theta=& \int_{\partial D^2}\mathscr{H}(v)\varphi \frac{1}{iz}dz=\int_{D^2}\frac{\partial}{\partial \overline{z}}\left(\mathscr{H}(v)\varphi\frac{1}{iz}\right)d\overline{z}\wedge dz\\
    \nonumber
    =&2\int_{D^2}\frac{1}{z}\mathscr{H}(v)\frac{\partial\varphi}{\partial \overline{z}} dx^2
\end{align}
\begin{defn}\label{def: frac-Hopf-differential}
    For any $u\in H^\frac{1}{2}(\mathbb{S}^1)$, the \textbf{fractional Hopf differential of $u$} is the distribution given by
    \begin{align}
        \mathscr{H}_\frac{1}{2}(u)[\varphi]=2\int_{D^2}\frac{1}{z}\mathscr{H}(\tilde{u})\frac{\partial\tilde{\varphi}}{\partial \overline{z}} dx^2,
    \end{align}
    where $\tilde{u}$ and $\tilde{\varphi}$ denote the harmonic extensions in $D^2$ of $u$ and $\varphi$ respectively.
\end{defn}
Given $u\in H^\frac{1}{2}(\mathbb{S}^1)$, by the last inequality in computation \eqref{eq: computation-trace-Hopf}, $\mathscr{H}_\frac{1}{2}(u)[\varphi]$ does not change if we modify $\tilde{\varphi}$ smoothly inside of $D^2$. With the help of this observation we see that the map
\begin{align}
    \mathscr{H}_\frac{1}{2}: H^\frac{1}{2}(\mathbb{S}^1)\to\mathscr{D}',\quad u\mapsto \mathscr{H}_\frac{1}{2}(u)
\end{align}
is continuous. By computation \eqref{eq: computation-trace-Hopf}, if $u\in C^\infty(\mathbb{S}^1)$ and $\tilde{u}$ denotes its harmonic extension in $D^2$, then the distribution $\mathscr{H}_\frac{1}{2}(u)$ is represented by the trace of $\mathscr{H}_\frac{1}{2}(\tilde{u})$ on $\partial D^2$. Therefore the map $\mathscr{H}_\frac{1}{2}$ extends continuously the operator assigning to any smooth $u$ the trace of the Hopf differential of its harmonic extension on $\partial D^2$.\\
Next we will see how the fractional Hopf differential $\mathscr{H}_\frac{1}{2}$ is related to the inner variations of $E_\frac{1}{2}$.
\begin{prop}\label{prop: frac-Hopf-diff}
    Let $u\in H^\frac{1}{2}(\mathbb{S}^1)$. Then
    \begin{align}\label{eq: frac-Hopf-diff-variation-statement}
        \mathscr{H}_\frac{1}{2}(u)=\frac{e^{-i2\theta}}{2i}\left(\mathcal{V}_\frac{1}{2}(u)+iH\left(\mathcal{V}_\frac{1}{2}(u)\right)\right).
    \end{align}
    In particular $u$ is a stationary point of $E_\frac{1}{2}$ if and only if $\mathscr{H}_\frac{1}{2}(u)=0$.
\end{prop}
Here $H$ denotes the Hilbert transform, defined through the following property: for any $f\in \mathscr{D}'(\mathbb{S}^1)$, for any $n\in \mathbb{Z}$
\begin{align}
    \widehat{H(f)}(n)=-i\operatorname{sgn}(n)\widehat{f}(n).
\end{align}
\begin{proof}
    We will first prove the result for a smooth map $u$. Denote $\tilde{u}$ the harmonic extension of $u$. Then by \eqref{eq_comm_Hopf_diff_in_polar_coordinates} the restriction of the Hopf differential of $\tilde{u}$ to $\partial D^2$ is given by
    \begin{align}
        \mathscr{H}(\tilde{u})\vert_{\partial D^2}=\frac{\overline{z}^2}{4}\left(\lvert(-\Delta)^\frac{1}{2}u\rvert^2-\lvert\partial_\theta u\rvert^2-2i(-\Delta)^\frac{1}{2}u\cdot u'\right).
    \end{align}
    Since $\tilde{u}$ is harmonic, $\mathscr{H}(\tilde{u})$ is holomorphic, therefore $h(z):=i4z^2\mathscr{H}(\tilde{u})(z)$ is also holomorphic.
    In particular, the harmonic extensions of $2(-\Delta)^\frac{1}{2}u\cdot u'$ and $\lvert(-\Delta)^\frac{1}{2}u\rvert^2-\lvert\partial_\theta u\rvert^2$ are harmonic conjugate (they correspond respectively to the real and to the imaginary part of the holomorphic function $h$), therefore we have
    \begin{align}
        H(2(-\Delta)^\frac{1}{2}u\cdot u')=\lvert(-\Delta)^\frac{1}{2}u\rvert^2-\lvert\partial_\theta u\rvert^2.
    \end{align}
    Since for smooth functions the fractional Hopf differential coincides with the trace of the harmonic extension, there holds
    \begin{align}\label{eq: frac-Hopf-diff-variation}
        \mathscr{H}_\frac{1}{2}(u)=\frac{e^{-i2\theta}}{2}\left(H\left((-\Delta)^\frac{1}{2}u\cdot u'\right)-i(-\Delta)^\frac{1}{2}u\cdot u'\right)=\frac{e^{-i2\theta}}{2i}\left(\mathcal{V}_\frac{1}{2}(u)+iH\left(\mathcal{V}_\frac{1}{2}(u)\right)\right).
    \end{align}
    By continuity, identity \eqref{eq: frac-Hopf-diff-variation} extends to any $u\in H^\frac{1}{2}(\mathbb{S}^1)$.
    Now if $u\in H^\frac{1}{2}(\mathbb{S}^1)$ is a stationary point of $E_\frac{1}{2}$ then $\mathcal{V}_\frac{1}{2}(u)=0$ and clearly $\mathscr{H}_\frac{1}{2}(u)=0$. On the other hand if $\mathscr{H}_\frac{1}{2}(u)=0$ we have $H\left(\mathcal{V}_\frac{1}{2}(u)\right)=-\mathcal{V}_\frac{1}{2}(u)$, i.e. $\mathcal{V}_\frac{1}{2}(u)$ consists only of negative frequencies, but since $\mathcal{V}_\frac{1}{2}(u)$ is real valued, we have $\mathcal{V}_\frac{1}{2}(u)=0$, so that $u$ is a stationary point of $E_\frac{1}{2}$.
\end{proof}
Next we deduce from Proposition \ref{prop: frac-Hopf-diff} a characterization of stationary points of $E_\frac{1}{2}$ in terms of Fourier coefficients. This result was already obtained in \cite{DaLio-Pohozaev} (Proposition 1.2) and \cite{Cabre-Mas}.
\begin{lem}
    Let $u\in H^\frac{1}{2}(\mathbb{S}^1, \mathbb{R}^k)$. Then $u$ is a stationary point of $E_\frac{1}{2}$ if and only if for any $k\in \mathbb{N}$ there holds
    \begin{align}\label{eq: proof-characterization-stationarity-Fourier}
        \sum_{\substack{m,n\in\mathbb{N}\\ m+n=k}}m n\,\widehat{u}(m)\cdot \widehat{u}(n)=0,
    \end{align}
    where $\cdot$ denotes the non-hermitian dot product in $\mathbb{C}^k$.
\end{lem}
\begin{proof}
    First we claim that for any $u\in C^\infty(\mathbb{S}^1)$ we have\footnote{We will denote alternatively by $\widehat{f}(n)$ or $\mathscr{F}(f)(n)$ the $n$-th Fourier coefficient of a distribution $f$, defined as
    \begin{align*}
        \widehat{f}(n):=\frac{1}{2\pi}\langle f, e^{-inx}\rangle.
    \end{align*}}
    \begin{align}\label{eq: proof-characterization-stationarity-Fourier}
        \mathscr{F}\left(2e^{-i2\theta}\mathscr{H}_\frac{1}{2}(u)\right)(k)=\sum_{\substack{m,n\in\mathbb{N}\\ m+n=k}}m n\,\widehat{u}(m)\cdot \widehat{u}(n),
    \end{align}
    while the negative Fourier coefficients are zero.\\
    Denote by $u_+$, $u_-$ the positive- and non positive-frequencies part of $u$ respectively. 
    Note that
    \begin{align}
        \mathcal{V}_\frac{1}{2}(u)=&(-\Delta)^\frac{1}{2}(u_++u_-)\cdot i(-\Delta)^\frac{1}{2}(u_+-u_-)\\
        \nonumber
        =&i(-\Delta)^\frac{1}{2}u_+\cdot(-\Delta)^\frac{1}{2}u_+-i(-\Delta)^\frac{1}{2}u_-\cdot(-\Delta)^\frac{1}{2}u_-
    \end{align}
Notice that the first term on the right hand side consists only of positive frequencies, while the second consists only of negative frequencies.
Thus
\begin{align}
    \mathcal{V}_\frac{1}{2}(u)+iH\left(\mathcal{V}_\frac{1}{2}(u)\right)=2i(-\Delta)^\frac{1}{2}u_+\cdot(-\Delta)^\frac{1}{2}u_+
\end{align}
    and by Proposition \ref{prop: frac-Hopf-diff}
    \begin{align}
        \mathscr{H}_\frac{1}{2}(u)=e^{-i2\theta}(-\Delta)^\frac{1}{2}u_+\cdot(-\Delta)^\frac{1}{2}u_+.
    \end{align}
    Since $u_+$ consists only of positive frequencies, for any $k\in \mathbb{N}$ we have
    \begin{align}
        \mathscr{F}\left(e^{-i2\theta}\mathscr{H}_\frac{1}{2}(u)\right)(k)=&\sum_{n\in \mathbb{Z}}\widehat{(-\Delta)^\frac{1}{2}u_+}(k-n)\cdot\widehat{(-\Delta)^\frac{1}{2}u_+}(n)\\
        \nonumber
        =&\sum_{\substack{m,n\in\mathbb{N}\\ m+n=k}}m n\,\widehat{u}(m)\cdot \widehat{u}(n)
    \end{align}
    while the the negative coefficients are zero. This conclude the proof of the claim. By approximation, \eqref{eq: proof-characterization-stationarity-Fourier} holds for any $u\in H^\frac{1}{2}(\mathbb{S}^1)$.
    Since $u\in H^\frac{1}{2}(\mathbb{S}^1)$ is a stationary point of $E_\frac{1}{2}$ if and only if $\mathscr{H}_\frac{1}{2}(u)=0$, the result follows.
\end{proof}
For $k=2$, the previous result yields a simple relation for the first Fourier coefficient, which is reminiscent of the balancing condition \eqref{eq: balancing condition} for harmonic functions. The following result already appeared in \cite{dalio2016pohozaevtype}
\begin{cor}
    For any stationary point $u\in H^\frac{1}{2}(\mathbb{S}^1)$ of $E_\frac{1}{2}$ there holds
    \begin{align}\label{eq: balancing condition lemma}
        \left\lvert\int_0^{2\pi}u(\theta)\cos(\theta)d\theta\right\rvert^2=\left\lvert\int_0^{2\pi}u(\theta)\sin(\theta)d\theta\right\rvert^2
    \end{align}
    and
    \begin{align}\label{eq: balancing condition imaginary}
        \left(\int_0^{2\pi}u(\theta)\cos(\theta)d\theta\right)\cdot\left(\int_0^{2\pi}u(\theta)\cos(\theta)d\theta\right)=0.
    \end{align}
\end{cor}
\begin{proof}
    Let $u\in H^\frac{1}{2}(\partial D^2)$ be a stationary point of $E_\frac{1}{2}$. By \eqref{eq: proof-characterization-stationarity-Fourier} with $k=2$ we have
    \begin{align}
        \widehat{u}(1)\cdot\widehat{u}(1)=0,
    \end{align}
    where $\cdot$ denotes the non-hermitian dot product in $\mathbb{C}^k$. Therefore we have
    \begin{align}\label{eq: derivation balancing condition}
        0=\sum_{j=1}^k&\left(\int_{0}^{2\pi}u^j(\theta)\cos(\theta)d\theta+i\int_{0}^{2\pi}u^j(\theta)\sin(\theta)d\theta\right)^2\\
        \nonumber
        =\sum_{j=1}^k&\left(\left(\int_{0}^{2\pi}u^j(\theta)\cos(\theta)d\theta\right)^2-\left(\int_{0}^{2\pi}u^j(\theta)\sin(\theta)d\theta\right)^2\right.\\
        \nonumber
        &\left.+2i\left(\int_{0}^{2\pi}u^j(\theta)\cos(\theta)d\theta\right)\left(\int_{0}^{2\pi}u^j(\theta)\sin(\theta)d\theta\right)\right).
    \end{align}
    Considering separately the real and imaginary parts of \eqref{eq: derivation balancing condition} we obtain respectively \eqref{eq: balancing condition lemma} and \eqref{eq: balancing condition imaginary}.
\end{proof}
Finally, with the help of the fractional Hopf differential we prove the following theorem, which describes the relationship between conformal harmonic maps and stationary points of $E_\frac{1}{2}$.
\begin{thm}
    Let $u\in H^\frac{1}{2}(\mathbb{S}^1)$ and let $\tilde{u}\in H^1(D^2)$ denote its harmonic extension. Then the following are equivalent.
    \begin{enumerate}
        \item $u$ is a stationary point of $E_\frac{1}{2}$, i.e. $\mathscr{H}_\frac{1}{2}(u)=0$,
        \item $\tilde{u}$ is weakly conformal, i.e. $\mathscr{H}(\tilde{u})=0$.
    \end{enumerate}
\end{thm}
\begin{proof}
    If $\tilde{u}$ is conformal, then $\mathscr{H}(\tilde{u})=0$ and thus by Definition \ref{def: frac-Hopf-differential} we have $\mathscr{H}_\frac{1}{2}(u)=0$.\\
    Assume next that $\mathscr{H}_\frac{1}{2}(\tilde{u})=0$. We would like to say that since $\mathscr{H}(\tilde{u})$ is holomorphic, if its trace on $\partial D^2$ vanishes then $\mathscr{H}(\tilde{u})$ vanishes as well. However $\mathscr{H}_\frac{1}{2}(u)$ coincides with the trace of $\mathscr{H}(\tilde{u})$ only if $u$ is sufficiently smooth, so we need to be more careful. Recall that
    \begin{align}
        \tilde{u}(r,\theta)=\sum_{n\in \mathbb{Z}}r^{\lvert n\rvert}e^{in\theta}\widehat{u}(n).
    \end{align}
    For any $r\in (0,1), \theta\in [0,2\pi)$ set $u_r(\theta)=u(r,\theta)$. Then for any $n\in \mathbb{Z}$ we have $\widehat{u_r}(n)=r^{\lvert n\rvert}\hat{u}(n)$. In particular $\eqref{eq: proof-characterization-stationarity-Fourier}$ implies that for any $n\in \mathbb{Z}$
    \begin{align}
        \mathscr{F}\left(e^{-i2\theta}\mathscr{H}_\frac{1}{2}(u)\right)(n)=r^{\lvert n\rvert}\mathscr{F}\left(2e^{-i2\theta}\mathscr{H}_\frac{1}{2}(u_r)\right)(n)
    \end{align}
    so that $\mathscr{H}_\frac{1}{2}(u)=0$ if and only if $\mathscr{H}_\frac{1}{2}(u_r)=0$.
    Denote $\widetilde{u_r}$ the harmonic extension of $u_r$ in $D^2$. Then $\mathscr{H}(\widetilde{u_r})$ is an holomorphic function on $D^2$, and since $u_r$ is smooth up to the boundary and $\mathscr{H}_\frac{1}{2}(u_r)=0$ we conclude that $\mathscr{H}(\widetilde{u_r})=0$.
    Finally we notice that by uniqueness of the harmonic extension, $\widetilde{u_r}(s,\theta)=\tilde{u}(rs,\theta)$. Therefore $\mathscr{H}(\widetilde{u_r})(x)=r^2\mathscr{H}(\tilde{u})(rx)$ in $D^2$. As the argument above is valid for any $r\in (0,1)$, we conclude that $\mathscr{H}(\tilde{u})=0$, i.e. $\tilde{u}$ is weakly conformal.    
\end{proof}
\begin{rem}
    As noted in the introduction, the fact that the harmonic extension of a half harmonic map is conformal has been previously established (see \cite{DALIO-compactness}, \cite{Millot-Sire} and \cite{DaLio-Martinazzi-Riviere}). The proof of this fact, which is also based on the holomorphicity of the Hopf differential, relies on the smoothness of half harmonic maps (see Theorem 1.8 in \cite{Sub-criticality}). By introducing the fractional Hopf differential, this argument is extended to functions $u$ in $H^\frac{1}{2}(\mathbb{S}^1)$.
    Loosely speaking, the result above says that the map $u$ has a conformal harmonic extension not only if $P_M^T (-\Delta)^\frac{1}{2}u=0$ (where $P_M^T$ denotes the orthogonal projection on the tangent space of the target manifold $M$), but also when the less restrictive condition $(-\Delta)^\frac{1}{2}u\cdot u'=0$ is met. Moreover, this latter condition is shown to be necessary.
\end{rem}

\section{A Noether theorem for the half Dirichlet energy}
Whenever a Lagrangian (or more precisely its energy density) is invariant under a smooth family of variations in the domain, Noether's theorem \cite{Noether1918} allows to derive a conservation law for stationary points of the Lagrangian.
With the help of the characterization of stationary points of the half Dirichlet energy obtained above we will derive a Noether theorem for its symmetries.\\
Notice that by Lemma \ref{lem: estimate div_1/commutator}, the distribution $D_\frac{1}{2}(a,b)$ is well defined (by continuous extension) even when $a\in H^{-\frac{1}{2}}$, $b\in H^\frac{1}{2}$.\\
\begin{thm}\label{prop_final_version_Noether_thm_vardom_12}
Let $X$ be a smooth vector field on $\mathbb{S}^1$ and assume that for any $v\in C^\infty(\mathbb{S}^1,\mathbb{R}^m)$ its flow $\phi_t$ satisfies
\begin{equation}\label{eq_diffeo_is_conformal12_2}
v(\phi_t(x))\cdot(-\Delta)^\frac{1}{2}(v\circ\phi_t)(x)=v(\phi_t(x))\cdot(-\Delta)^\frac{1}{2}v(\phi_t(x))\phi_t'(x)
\end{equation}
for $x\in \mathbb{S}^1$ and $t$ in a neighbourhood of the origin.
Assume that $u\in H^\frac{1}{2}(\mathbb{S}^1, \mathbb{R}^m)$ is a stationary point of $E_\frac{1}{2}$.
Then we have
\begin{equation}\label{eq_prop_result_Noether_thm_vardom_12}
\frac{d}{dx}\left((-\Delta)^\frac{1}{2}u \cdot uX\right)=D_\frac{1}{2}(u'X,u).
\end{equation}
\end{thm}
\begin{proof}
First we will prove that for any $u\in C^\infty(\mathbb{S}^1)$ there holds
\begin{align}
    \frac{d}{dx}\left(X(-\Delta)^\frac{1}{2}u\cdot u\right)=2(-\Delta)^\frac{1}{2}u\cdot u' X+
    D_\frac{1}{2}(u'X,u).
\end{align}
Observe that by Taylor's Theorem
\begin{equation}\label{eq_approx_formula_vardom}
u\circ\phi_t=u+tu'X+o_{C^1}(t)
\end{equation}
as $t\to 0$ (in the following, the $o$ notation will always refer to $t\to 0$).
Thus we have
\begin{equation}\label{eq_fraclap_comp_diffeo_step14_1}
(-\Delta)^\frac{1}{2} (u\circ\phi_t)=(-\Delta)^\frac{1}{2}u+t(-\Delta)^\frac{1}{2}(u'X)+o_{L^2}(t).
\end{equation}
Observe that since $u\in C^\infty(\mathbb{S}^1)$, $(-\Delta)^\frac{1}{2}u\in C^\infty(\mathbb{S}^1)$. Thus, by Taylor's Theorem,
\begin{equation}\label{eq_fraclap_comp_diffeo_step14_3}
(-\Delta)^\frac{1}{2}u(\phi_t(x))=(-\Delta)^\frac{1}{2}u(x)+t(-\Delta)^\frac{1}{2}u'(x)\ X(x)+o_{L^\infty}(t).
\end{equation}
Therefore, by (\ref{eq_fraclap_comp_diffeo_step14_1}) and  (\ref{eq_fraclap_comp_diffeo_step14_3}), for $t\in \mathbb{R}$, $x\in \mathbb{S}^1$
\begin{align}\label{eq: for derivative}
    (-\Delta)^\frac{1}{2}(u\circ\phi_t)(x)=(-\Delta)^\frac{1}{2}u(\phi_t(x))+t\left((-\Delta)^\frac{1}{2}(u'X)-(-\Delta)^\frac{1}{2}u'X\right)+o_{L^2}(t).
\end{align}
Recall that by assumption \eqref{eq_diffeo_is_conformal12_2} for any $x\in \mathbb{S}^1$ we have
\begin{equation}\label{eq_diffeo_is_conformal14}
(-\Delta)^\frac{1}{2}(u\circ\phi_t)\cdot u\circ\phi_t=((-\Delta)^\frac{1}{2}u\cdot u)(\phi_t(x))\phi_t'(x).
\end{equation}
If we derive (\ref{eq_diffeo_is_conformal14}) with respect to $t$ and evaluate in $t=0$ we obtain
\begin{equation}
\begin{split}\label{eq_derivative_diffeo_is_conformal14}
\frac{d}{dt}\bigg\vert_{t=0}(-\Delta)^\frac{1}{2}(u\circ\phi_t)\cdot u\circ\phi_t=&\frac{d}{dt}\bigg\vert_{t=0}((-\Delta)^\frac{1}{2}u \cdot u)(\phi_t(x))+(-\Delta)^\frac{1}{2}u \cdot u X'(x)\\
=&\frac{d}{dx}((-\Delta)^\frac{1}{2}u\cdot u)X(x)+(-\Delta)^\frac{1}{2}u\cdot u X'(x)\\
=&\frac{d}{dx}\left((-\Delta)^\frac{1}{2}u\cdot u X\right)(x).
\end{split}
\end{equation}
On the other hand, it follows from \eqref{eq: for derivative} that for almost any $x\in \mathbb{S}^1$
\begin{align}\label{eq_derivative_of_f14}
    \frac{d}{dt}\bigg\vert_{t=0}(-\Delta)^\frac{1}{2}(u\circ\phi_t)\cdot u\circ\phi_t=&\frac{d}{dx}\left(u\cdot (-\Delta)^\frac{1}{2}u\right)X\\
    \nonumber
    &+u\cdot\left((-\Delta)^\frac{1}{2}(u'X)-(-\Delta)^\frac{1}{2}u' X\right).
\end{align}
Comparing \eqref{eq_derivative_diffeo_is_conformal14} and \eqref{eq_derivative_of_f14} we obtain
\begin{align}\label{eq: formula-for Pohozaev}
    u\cdot\frac{d}{dx}\left(X(-\Delta)^\frac{1}{2}u\right)=u\cdot(-\Delta)^\frac{1}{2}(u'X),
\end{align}
which can be rewritten as
\begin{align}
    \frac{d}{dx}\left(X(-\Delta)^\frac{1}{2}u\cdot u\right)=2(-\Delta)^\frac{1}{2}u\cdot u' X+
    D_\frac{1}{2}(u'X,u).
\end{align}
This conclude the proof of the claim.\\
Next suppose that $u\in H^\frac{1}{2}$ is a stationary point of $E_\frac{1}{2}$. Then there exists a sequence $\{u_n\}_{n\in \mathbb{N}}$ of smooth functions converging to $u$ in $H^\frac{1}{2}$.
By Proposition \ref{prop: first-variation-formula} we have
\begin{align}
    \lim_{n\to \infty}(-\Delta)^\frac{1}{2}u_n u_n' X= \lim_{n\to \infty}\mathcal{V}_\frac{1}{2}(u_n)X=\mathcal{V}_\frac{1}{2}(u)X=0,
\end{align}
where the limit are taken in the space of distributions.
On the other hand by Lemma \ref{lem: estimate div_1/commutator}
\begin{align}
    &\frac{d}{dx}\left((-\Delta)^\frac{1}{2}u uX\right)-D_\frac{1}{2}(u'X,u)\\
    \nonumber
    =&\lim_{n\to\infty}\left(\frac{d}{dx}((-\Delta)^\frac{1}{2}u_n u_nX)+D_\frac{1}{2}(u_n'X,u_n)\right)\\
    \nonumber
    =&\lim_{n\to\infty}2(-\Delta)^\frac{1}{2}u_n u_n X=0.
\end{align}
This completes the proof of the theorem.
\end{proof}

We remark that it would seem more natural to look for Noether theorems for $E_\frac{1}{2}$ in terms of the energy density $\left\lvert(-\Delta)^\frac{1}{4}u\right\rvert^2$. In fact one can obtain with similar methods the following theorem (see \cite{gaia-MT}, Theorem 4.5) for a family of energies which includes the half Dirichlet energy. However, it seems that it is more difficult to find vector fields $X$ for which $\lvert(-\Delta)^\frac{1}{4}u\rvert^2$ is invariant and for which Noether theorems delivers nontrivial results.\\
Let $m\in\mathbb{N}_{>0}$. Let $L\in C^2(\mathbb{R}^m\times\mathbb{R}^m, \mathbb{R})$ so that for any $x,p\in\mathbb{R}^m$
\begin{equation}\label{eq_well_def_energy_vardom}
\lvert L(x,p)\rvert\leq C(1+\lvert p\rvert^2)
\end{equation}
for some $C>0$. For any $u\in H^1(\mathbb{S}^1,\mathbb{R}^m)$ we set
\begin{equation}\label{eq_def_1_4_Dirichlet_energy}
E(u):=\int_{\mathbb{S}^1}L\left(u(x),(-\Delta)^\frac{1}{4}u(x)\right)dx.
\end{equation}
\begin{thm}\label{Thm: Noether-domain}
Let $X$ be a smooth vector field on $\mathbb{S}^1$ and assume that $L$ is invariant with respect to the variation generated by $X$, i.e.
\begin{equation}
L(u\circ\phi_t, (-\Delta)^\frac{1}{4} (u\circ\phi_t))(x)=L(u, (-\Delta)^\frac{1}{4} u)\circ \phi_t(x)\lvert D\phi_t(x)\rvert
\end{equation}
Assume that $u\in H^\frac{1}{2}(\mathbb{S}^1, \mathbb{R}^m)$ satisfies the following stationarity equation:
\begin{equation}\label{eq_prop_stationary_equation_14}
\frac{d}{dx}\left(L\left(u, (-\Delta)^\frac{1}{4}u\right)\right)X+D_\frac{1}{4}\left(D_pL\left(u, (-\Delta)^\frac{1}{4}u\right), u'\right)X=0.
\end{equation}
Then
\begin{equation}\label{eq_prop_result_Noether_thm_vardom_14}
\frac{d}{dx}\left(L\left(u, (-\Delta)^\frac{1}{4}u\right)X\right)+D_\frac{1}{4}\left(D_pL\left(u, (-\Delta)^\frac{1}{4}u\right), u'X\right)=0.
\end{equation}
\end{thm}

To give an application of Theorem \ref{prop_final_version_Noether_thm_vardom_12}, we show that the energy density $u\cdot (-\Delta)^\frac{1}{2}u$ is invariant under conformal transformations of the disc.\\
To this end recall that given any $u\in C^\infty(\mathbb{S}^1)$, its harmonic extension $\tilde{u}$ in $D^2$ is given by
\begin{align}
    \tilde{u}(r, \theta)=\sum_{n\in \mathbb{Z}}r^{\lvert n\rvert}e^{in\theta}\widehat{u}(n),
\end{align}
so that
\begin{align}
    \partial_r\tilde{u}(r, \theta)=\sum_{n\in \mathbb{Z}}\lvert n\rvert r^{\lvert n\rvert-1}e^{in\theta}\widehat{u}(n).
\end{align}
In particular we see that the half Laplacian $(-\Delta)^\frac{1}{2}$ coincides with the operator mapping a smooth function on $\mathbb{S}^1$ to the radial derivative of its harmonic extension on $\mathbb{S}^1$.\\
We also recall that the group of holomorphic diffeomorphisms of $D^2$, i.e. the M\"obius group $\mathscr{M}(D^2)$, consists of all the maps of the following form
\begin{equation}\label{eq_general_form_Mobius2}
    m_{a,\mu}:\,D^2\to D^2,\quad z\mapsto\mu\frac{z-a}{\bar{a}z-1},
\end{equation}
for $\mu$, $a\in \mathbb{C}$ with $\lvert \mu\rvert=1$, $\lvert a\rvert<1$.
\begin{lem}\label{lem: conformal-invariance}
    Let $\phi$ be a holomorphic diffeomorphism of $D^2$. Then for any $u\in C^\infty(\mathbb{S}^1)$ there holds
    \begin{align}
        (-\Delta)^\frac{1}{2}(u\circ\phi)(x)=\lvert\partial_\theta\phi (x)\rvert(-\Delta)^\frac{1}{2}u(\phi(x))
    \end{align}
    on $\mathbb{S}^1$.
    In particular, if $\{\phi_t\}_{t\in \mathbb{R}}$ is a smooth family of diffeomorphisms of $D^2$, then it satisfies \eqref{eq_diffeo_is_conformal12_2}.
\end{lem}
\begin{proof}
Let $u\in C^\infty(\mathbb{S}^1)$, let $\tilde{u}$ be its harmonic extension in $D^2$ and let $\phi\in \mathscr{M}(D^2)$. Notice that since $\phi$ is of the form \eqref{eq_general_form_Mobius2}, it extends to a diffeomorphism of $\overline{D^2}$, which we still denote $\phi$.
We observe that $\tilde{u}\circ \phi$ is harmonic as composition of an harmonic map with an holomorphic map.
Now, for any $r\in (0,1]$
\begin{equation}
\partial_r (\tilde{u}\circ\phi)=D\tilde{u}\circ \phi\cdot\partial_r\phi.
\end{equation}
We also observe that since $\phi$ restricts to a diffeomorphism of $\mathbb{S}^1$ and it is conformal and positively oriented, for any $\theta\in \mathbb{S}^1$
\begin{equation}
\partial_r \phi(1,\theta)=\frac{\lvert \partial_r \phi(1,\theta)\rvert}{\lvert \phi(1,\theta)\rvert}\phi(1,\theta)=\lvert \partial_r \phi (1, \theta)\rvert\phi(1, \theta).
\end{equation}
Therefore for any $\theta\in \mathbb{S}^1$
\begin{align}\label{eq_lim_boundary_smooth_case}
(-\Delta)^\frac{1}{2}(u\circ \phi)(1,\theta)=& D\tilde{u}(\phi(1,\theta))\cdot\partial_r\phi(1,\theta)=
D\tilde{u}(\phi(1,\theta))\cdot \phi(1, \theta)\lvert \partial_r \phi (1, \theta)\rvert\\
\nonumber
=&(-\Delta)^\frac{1}{2}u(\phi(1,\theta))\lvert\partial_\theta\phi(1,\theta)\rvert,
\end{align}
where in the last step we used the fact that $\lvert\frac{1}{r} \partial_\theta \phi\rvert=\lvert\partial_r\phi\rvert$ on $\mathbb{S}^1$ as $\phi$ is holomorphic.
\end{proof}
As a concrete example, let's first consider rotations of the disc: let $a=1$ and $\mu=e^{it}$ for $t\in \mathbb{R}$. On $\mathbb{S}^1$, thought of as $\quot{\mathbb{R}}{2\pi\mathbb{Z}}$, this corresponds to the flow $\phi_t(x)=x+t$ for $x\in \mathbb{S}^1$, $t\in\mathbb{R}$. This is the flow generated by $X\equiv1$. Therefore, in this case, (\ref{eq_prop_result_Noether_thm_vardom_14}) is equivalent to (\ref{eq_prop_stationary_equation_14}).\\
Now let's consider the family of diffeomorphisms corresponding to $\mu=1$, $a_t=e^{i\delta}t$ for some $\delta\in [0,1)$ and $t\in \mathbb{R}$ in a neighbourhood of $0$, i.e.
\begin{align}
    {\phi}_t(r,\theta)=\frac{re^{i\theta}-te^{i\delta}}{r t e^{i(\theta-\delta)}-1}
\end{align}
On $\mathbb{S}^1\simeq\partial D^2$, this corresponds to the flow of the vector field 
\begin{equation}\label{eq: dilations-disc}
Y(e^{i\theta})=e^{i\delta}-e^{i(2\theta-\delta)}=2ie^{i\theta}\sin(\delta-\theta)
\end{equation}
for any $\theta\in [0,2\pi)$. Thus, on $\mathbb{S}^1\simeq \quot{\mathbb{R}}{2\pi\mathbb{Z}}$, $Y$ takes the form $Y(x)=2\sin(\delta-x)$ for any $x\in \mathbb{S}^1$. Geometrically, $Y$ induces a "dialation on $\mathbb{S}^1$" around the point $e^{i\delta}$. For the vector field $Y$, Theorem \ref{prop_final_version_Noether_thm_vardom_12} yields the following result.
\begin{cor}\label{cor: conservation-law}
For any stationary point $u\in H^\frac{1}{2}(\mathbb{S}^1)$ of $E_\frac{1}{2}$, for any $\delta\in [0,2\pi)$ there holds
\begin{equation}\label{eq_discussion_traces_of_conformal_maps_conserved_quantity_moving_a_2}
    \frac{d}{dx}\left(u\cdot(-\Delta)^\frac{1}{2}u\sin(\delta-x)\right)=D_\frac{1}{2}\left(u'\sin(\delta-x), u\right).
\end{equation}
\end{cor}
Finally we remark that the invariance properties of the half Dirichlet energy imply the following integral relation.
\begin{lem}\label{cor_Pohozaev_identities_on_the_circle}
    Let $X$ be a vector field on $\mathbb{S}^1$ whose flow consists of traces of holomorphic diffeomorphisms of $D^2$.
Then for any $u\in H^\frac{1}{2}(\mathbb{S}^1)$ there holds
\begin{equation}\label{eq_prop_Pohozaev_pointwise_circle}
\langle \mathcal{V}_\frac{1}{2}(u), X\rangle=0.
\end{equation}
\end{lem}

\begin{proof}
    Let $u\in H^\frac{1}{2}(\mathbb{S}^1)$. By Lemma \ref{lem: conformal-invariance}, for any conformal diffoemorphism $\phi$ there holds \eqref{eq_diffeo_is_conformal12_2}.
    Therefore a change of variables yields
    \begin{align}
        E_\frac{1}{2}(u\circ\phi)=E_\frac{1}{2}(u).
    \end{align}
    Thus for any vector field $X$ as in the lemma we have
    \begin{align}
        \langle \mathcal{V}_\frac{1}{2}(u), X\rangle=\frac{d}{dt}\bigg\vert_{t=0}E_\frac{1}{2}(u\circ\phi_t)=0,
    \end{align}
    where $\phi_t$ denotes the flow of $X$.\\
    Alternatively, the Lemma can be obtained integrating over $\mathbb{S}^1$ identity \eqref{eq: formula-for Pohozaev}.  
\end{proof}
\begin{rem}
    For the vector fields considered above, the lemma implies
\begin{equation}
\int_{\mathbb{S}^1}u'(x)(-\Delta)^\frac{1}{2}u(x)dx=0,
\end{equation}
and, for any $\delta\in \mathbb{S}^1$,
\begin{equation}\label{eq_prop_Pohozaev_circle_dilation_integral_form}
\int_{\mathbb{S}^1}u'(x)(-\Delta)^\frac{1}{2}u(x)\sin(\delta-x)dx=0
\end{equation}
for any $u\in H^1(\mathbb{S}^1)$.
\end{rem}
\begin{rem}
    It is also possible to obtain Noether theorems for $E_\frac{1}{2}$ for variations in the target, see \cite{gaia-MT}, Section 3. For instance, one can show that the invariance under rotations in the target of the energy density $u\cdot(-\Delta)^\frac{1}{2}u$ among maps taking values in a sphere $\mathbb{S}^k$ implies the following conservation law for $\frac{1}{2}$-harmonic maps:
    \begin{align}\label{eq: fractional-conservation-law-target}
        \operatorname{div}_\frac{1}{2}\left(u^j(x)\operatorname{d}_\frac{1}{2}u^i(x,y)-u^i(x)\operatorname{d}_\frac{1}{2}u^j(x,y)\right)=0\qquad \forall i,j\in \{1,...,k\}
    \end{align}
    (for the definition of the fractional gradient $\operatorname{d}_\frac{1}{2}$ see \cite{Mazowiecka-Schikorra}).
    This is the fractional counterpart to the following conservation for harmonic maps with values in a sphere $\mathbb{S}^k$:
    \begin{align}\label{eq: conservation-law-target}
        \operatorname{div}\left(u^i\nabla u^j-u^j\nabla u^i\right)=0\qquad\forall i,j\in \{1,...,k\},
    \end{align}
    which can also be obtained as a consequence of the invariance under rotation of the energy density $\lvert\nabla u\rvert^2$. \eqref{eq: conservation-law-target} plays a crucial role in F. H\'elein's proof of the regularity for weakly harmonic maps from a surface to a sphere (see \cite{Helein-regularity}). Similarly \eqref{eq: fractional-conservation-law-target} has been used in \cite{Mazowiecka-Schikorra} to obtain a new proof of the regularity of $\frac{1}{2}$-harmonic maps from $\mathbb{R}$ to a sphere.
\end{rem}

\appendix
\section{Commutator estimates and fractional divergences}
In this appendix we will derive some estimate for commutators, which will allow us to extend the definition of the operator $D_s(a,b)$ in a distributional sense to a wide family of maps $a,b$. We will then see how the operator $D_s(a,b)$ is related to the fractional divergence operator introduced in \cite{Mazowiecka-Schikorra}.\\
We will make use of the following function space.
\begin{defn}\label{defn: Wiener-algebra}
Let $\mathbb{A}(\mathbb{S}^1)$ be the space of all the elements $u$ in $\mathscr{D}'(\mathbb{S}^1)$ such that
\begin{equation}
\lVert u\rVert_{\mathbb{A}}:=\sum_{k\in\mathbb{Z}}\lvert \widehat{u}(k)\rvert<\infty.
\end{equation}
$\mathbb{A}(\mathbb{S}^1)$ is called \textbf{Wiener algebra}.
\end{defn}
\begin{rem}\label{rem_Wiener_algebra}
Recall that for any $s>\frac{1}{2}$ there is a continuous embedding $H^s(\mathbb{S}^1)\hookrightarrow \mathbb{A}(\mathbb{S}^1)$. Indeed, for any $u\in H^s(\mathbb{S}^1)$,\begin{equation}\label{eq_rem_embedding_Sobolev_in_Wiener}
\sum_{\substack{k\in\mathbb{Z}\\ k\neq 0}}\lvert \widehat{u}(k)\rvert\leq \left(\sum_{\substack{k\in\mathbb{Z}\\ k\neq 0}}\lvert \widehat{u}(k)\rvert^2\lvert k\rvert^{2s}\right)^\frac{1}{2}\left(\sum_{\substack{k\in\mathbb{Z}\\ k\neq 0}}\lvert k\rvert^{-2s}  \right)^\frac{1}{2}
\end{equation}
by Cauchy-Schwartz inequality. As $s>\frac{1}{2}$, the second factor is finite. Moreover $\lvert \widehat{u}(0)\rvert\lvert \leq\lVert u\rVert_{H^s}$.\\
In particular, all the results of this section have a slightly weaker formulation in terms of Sobolev spaces only.
\end{rem}
\begin{lem}\label{lem: estimate for div}
    Let $s\in \left(0,\frac{1}{2}\right)$. For any $a,b,\varphi\in C^\infty(\mathbb{S}^1)$ there holds
    \begin{align}
        \left\lvert D_s(a,b)\varphi\right\rvert\leq C\lVert \varphi'\rVert_{\mathbb{A}}\lVert a\rVert_{H^{2s-1}}\lVert b\rVert_{L^2}.
    \end{align}
\end{lem}
\begin{proof}
We define the action of the commutator $[(-\Delta)^s, a]$ on $\varphi$ as follows:
\begin{align}
    [(-\Delta)^s, a]\varphi=(-\Delta)^s(a\varphi)-a(-\Delta)^s\varphi.
\end{align}
To prove the Lemma it will be enough to show the following estimate:
    \begin{align}\label{eq: commutator-estimate-proof}
        \left\lVert  [(-\Delta)^s, a]\varphi\right\rVert_{L^2}\leq C\lVert \varphi'\rVert_{\mathbb{A}}\lVert a\rVert_{H^{2s-1}}.
    \end{align}
For any $n\in \mathbb{Z}$ we have
    \begin{equation}
\begin{split}
\mathscr{F}\left([(-\Delta)^s, a]\varphi\right)(n)=&\lvert n\rvert^{2s} \widehat{a\varphi}(n)-\sum_{k\in\mathbb{Z}}\lvert n-k\rvert^{2s} \widehat{a}(n-k)\widehat{\varphi}(k)\\
=&\lvert n \rvert^{2s}\sum_{k\in\mathbb{Z}}\widehat{a}(n-k)\widehat{\varphi}(k)-\sum_{k\in\mathbb{Z}}\lvert n-k\rvert^{2s} \widehat{a}(n-k)\widehat{\varphi}(k)\\
=&\sum_{k\in\mathbb{Z}}\left(\lvert n\rvert^{2s}-\lvert n-k\rvert^{2s}\right)\widehat{\varphi}(k)\widehat{a}(n-k).
\end{split}
\end{equation}
By Plancherel's identity
\begin{align}
  \left\lVert  [(-\Delta)^s, a]\varphi\right\rVert_{L^2}^2=4\pi^2\sum_{n\in \mathbb{Z}}\left\lvert\sum_{k\in\mathbb{Z}}\left(\lvert n\rvert^{2s}-\lvert n-k\rvert^{2s}\right)\widehat{\varphi}(k)\widehat{a}(n-k)\right\rvert^2.   
\end{align}
Observe that if $\lvert k\rvert\geq 2\lvert n\rvert>0$, there holds
$\frac{\lvert k\rvert}{2}\leq \lvert k+n\rvert$; thus
\begin{equation}
\begin{split}
\left\lvert \lvert k+n\rvert^{2s}-\lvert k\rvert^{2s}\right\rvert\leq& s\int_{\lvert k\rvert\wedge\lvert n+k\rvert}^{\lvert k\rvert\vee\lvert n+k\rvert}r^{2s-1}dr\leq s \lvert n\rvert\left(\frac{\lvert k\rvert}{2}\right)^{2s-1}\\
\leq& 3^{1-2s}s \lvert n\rvert\lvert k-n\rvert^{2s-1}.
\end{split}
\end{equation}
Therefore Young's Inequality yields
\begin{equation}\label{eq_proof_lemma_L2norm_Gphi_firstsum}
\begin{split}
&\sum_{k\in\mathbb{Z}}\left(\sum_{0<\lvert n\rvert\leq\frac{\lvert k\rvert}{2}}\left\lvert \widehat{\varphi}(n)\widehat{a}(k-n)\right\rvert\left\lvert\lvert k+n\rvert^{2s}-\lvert k\rvert^{2s}\right\rvert\right)^2\\
\leq&3^{2(1-2s)}s^2\sum_{k\in\mathbb{Z}}\left(\sum_{0<\lvert n\rvert\leq\frac{\lvert k\rvert}{2}}\left\lvert \widehat{\varphi}(n) n\right\rvert\left\lvert\widehat{a}(k-n)\right\rvert\lvert k-n\rvert^{2s-1}\right)^2\\
\leq& 3^{2(1-2s)}s^2\left(\sum_{n\in \mathbb{N}}\left\lvert \widehat{\varphi}(n)n\right\rvert\right)^2\left(\sum_{k\in\mathbb{Z}}\left\lvert \widehat{a}(k)\right\rvert^2\lvert k\rvert^{4s-2}\right)=3^{2(1-2s)}s^2\lVert\varphi'\rVert_{\mathbb{A}}^2[a]_{\dot{H}^{2s-1}}^2.
\end{split}
\end{equation}
On the other hand, we always have
\begin{equation}
\left\lvert\lvert n+k\rvert^{2s}-\lvert k\rvert^{2s}\right\rvert\leq \lvert n\rvert^{2s},
\end{equation}
and if $\lvert k\rvert<2\lvert n\rvert$, $\lvert k-n\rvert\leq 3\lvert n\rvert$.
Therefore
\begin{equation}\label{eq_proof_lemma_L2norm_Gphi_secondsum}
\begin{split}
&\sum_{k\in\mathbb{Z}}\left(\sum_{\substack{\frac{\lvert k\rvert}{2}<\lvert n\rvert}}\left\lvert \widehat{\varphi}(n)\widehat{a}(k-n)\right\rvert\left\lvert\lvert k+n\rvert^{2s}-\lvert k\rvert^{2s}\right\rvert\right)^2\\
\leq&\sum_{k\in\mathbb{Z}}\left(\sum_{\substack{\frac{\lvert k\rvert}{2}<\lvert n\rvert}}\left\lvert \widehat{\varphi}(n)\right\rvert\lvert n\rvert^{2s}\lvert k-n\rvert^{1-2s}\left\lvert\widehat{a}(k-n)\right\rvert\lvert k-n\rvert^{2s-1}\right)^2\\
\leq & 3^{2(1-2s)} \sum_{k\in\mathbb{Z}}\left(\sum_{n\in \mathbb{Z}}\left\lvert \widehat{\varphi}(n)\right\rvert\lvert n\rvert\left\lvert\widehat{a}(k-n)\right\rvert\lvert\lvert k-n\rvert^{2s-1}\right)^2\\
\leq&3^{2(1-2s)}\left(\sum_{n\in \mathbb{Z}}\left\lvert \widehat{\varphi}(n)n\right\rvert\right)^2\left(\sum_{k\in\mathbb{Z}}\left\lvert \widehat{a}(k)\right\rvert^2\lvert k\rvert^{4s-2}\right)=3^{2(1-2s)}\lVert \varphi'\rVert_{\mathbb{A}}^2[ a]_{\dot{H}^{2s-1}}^2,
\end{split}
\end{equation}
where the second-last steps follow again from Young's inequality.
Therefore, combining (\ref{eq_proof_lemma_L2norm_Gphi_firstsum}) and (\ref{eq_proof_lemma_L2norm_Gphi_secondsum}) we obtain
\begin{equation}\label{eq_proof_lemma_intermediate_step_expression_with_seminorm}
 \left\lVert  [(-\Delta)^s, a]\varphi\right\rVert_{L^2}^2\leq 4\pi^2 3^{2(1-2s)}(1+s^2)\lVert\varphi\rVert_{\mathbb{A}}^2[a]_{\dot{H}^{2s-1}}^2
\end{equation}
This concludes the proof of the Lemma.
\end{proof}

\begin{lem}\label{lem: estimate div_1/commutator}
    Let $a,b,\varphi\in C^\infty(\mathbb{S}^1)$. Then
\begin{align}\label{eq: commutator-estimate-1}
    \left\lvert D_\frac{1}{2}(a,b)\varphi\right\rvert\leq  C\lVert(-\Delta)^\frac{3}{4}\varphi\rVert_{\mathbb{A}}\lVert a\rVert_{H^{-\frac{1}{2}}} \lVert b\rVert_{H^\frac{1}{2}}.
\end{align}
\end{lem}
\begin{proof}
Proceeding as in the proof of Lemma \ref{lem: estimate for div} we see that it is enough to show the following estimate:
    \begin{align}\label{eq: commutator-estimate-proof}
        \lVert[(-\Delta)^\frac{1}{2},a]\varphi\rVert_{H^{-\frac{1}{2}}}\leq C\lVert(-\Delta)^\frac{3}{4}\varphi\rVert_{\mathbb{A}}\lVert a\rVert_{H^{-\frac{1}{2}}}.
    \end{align}
In term of Fourier coefficients, we would like to obtain a bound for
\begin{align}
    \lVert [(-\Delta)^\frac{1}{2},a]\varphi\rVert_{H^{-\frac{1}{2}}}^2=&\sum_{n\in \mathbb{Z}}\left\lvert\mathscr{F}\left([(-\Delta)^\frac{1}{2},a]\varphi\right)(n)\right\rvert^2(1+\lvert n\rvert^2)^{-\frac{1}{2}}\\
    \nonumber
    =&\sum_{n\in \mathbb{Z}}\left\lvert\sum_{k\in\mathbb{Z}}(\lvert n\rvert-\lvert n-k\rvert)(1+\lvert n\rvert^2)^{-\frac{1}{4}}\widehat{\varphi}(k)\widehat{a}(n-k)\right\rvert^2. 
\end{align}
Observe that if $\lvert n\rvert\geq 2\lvert k\rvert$, $\lvert n-k\rvert\leq \frac{3}{2}\lvert n\rvert$ and therefore
\begin{equation}
\frac{\lvert n\rvert-\lvert n-k\rvert}{(1+\lvert n\rvert^2)^\frac{1}{4}}\leq \left(\frac{3}{2}\right)^\frac{1}{2}\frac{\lvert k\rvert}{(1+\lvert n-k\rvert^2)^\frac{1}{4}}.
\end{equation}
Thus, by Young's inequality
\begin{equation}\label{eq_proof_lemma_approximation_12_step_1}
\begin{split}
&\sum_{n\in\mathbb{Z}}\left\lvert\sum_{\lvert k\rvert\leq \frac{\lvert n\rvert}{2}}\widehat{\varphi}(k)\widehat{a}(n-k)(\lvert n\rvert-\lvert n-k\rvert)(1+\lvert n\rvert^2)^{-\frac{1}{4}}\right\rvert^2\\
\leq& \frac{3}{2}\sum_{n\in \mathbb{Z}}\left(\sum_{k\in\mathbb{Z}}\lvert\widehat{\varphi}(k)\rvert\lvert k\rvert\lvert\widehat{a}(n-k)\rvert\lvert(1+\lvert n-k\rvert^2)^{-\frac{1}{4}}\rvert\right)^2\\
\leq&\frac{3}{2}\left(\sum_{k\in \mathbb{Z}}\lvert\hat{\varphi}(k)k\rvert\right)^2\sum_{k\in \mathbb{Z}}\lvert \hat{a}(k)\rvert^2(1+\lvert k\rvert^2)^{-\frac{1}{2}}=\frac{3}{2}\lVert\varphi'\rVert_{\mathbb{A}}^2\lVert a\rVert_{H^{-\frac{1}{2}}}^2.
\end{split}
\end{equation}
On the other hand, if $\lvert n\rvert<2 \lvert k\rvert$, $\lvert n-k\rvert\leq 3\lvert k\rvert$ and therefore
\begin{equation}
\frac{\lvert n\rvert -\lvert n-k\rvert}{\left(1+\lvert n\rvert^2\right)^\frac{1}{4}}\leq \frac{\lvert k\rvert}{(1+\lvert n-k\rvert^2)^\frac{1}{4}}\left(1+\lvert n-k\rvert^2\right)^\frac{1}{4}\leq \frac{2\lvert k\rvert^\frac{3}{2}}{\left(1+\lvert n-k\rvert^2\right)^\frac{1}{4}}.
\end{equation}
Thus, by Young's inequality
\begin{align}\label{eq_proof_lemma_approximation_12_step_2}
&\sum_{n\in\mathbb{Z}}\left\lvert\sum_{\lvert n\rvert<2\lvert k\rvert}\widehat{\varphi}(k)\widehat{a}(n-k)(\lvert n\rvert-\lvert n-k\rvert)(1+\lvert n\rvert^2)^{-\frac{1}{4}}\right\rvert^2\\
\nonumber
\leq & 4 \sum_{n\in\mathbb{Z}}\left(\sum_{\lvert n\rvert<2\lvert k\rvert}\lvert\widehat{\varphi}(k)\rvert\lvert k\rvert^\frac{3}{2}\lvert\widehat{a}(n-k)\rvert\lvert (1+\lvert n-k\rvert^2)^\frac{1}{4}\right)^2\\
\nonumber
\leq& 4\left(\sum_{k\in \mathbb{Z}}\lvert\hat{\varphi}(k)\rvert\lvert k\rvert^\frac{3}{2}\right)^2\sum_{k\in \mathbb{Z}}\lvert \hat{a}(k)\rvert^2(1+\lvert k\rvert^2)^{-\frac{1}{2}}\\
\nonumber
=&4\lVert(-\Delta)^\frac{3}{4}\varphi\rVert_{\mathbb{A}}^2\lVert a\rVert_{H^{-\frac{1}{2}}}^2=4 \lVert(-\Delta)^\frac{3}{4}\varphi\rVert_{\mathbb{A}}^2\lVert w\rVert_{H^{-\frac{1}{2}}}.
\end{align}
Combining (\ref{eq_proof_lemma_approximation_12_step_1}) and (\ref{eq_proof_lemma_approximation_12_step_2}) we obtain estimate \eqref{eq: commutator-estimate-proof}.
\end{proof}

Next we discuss the link of the operator $D_s$ with the fractional divergence operator introduced in \cite{Mazowiecka-Schikorra}
For $s\in \left(0,\frac{1}{2}\right)$ denote $K^s$ the kernel of the $s$-fractional Laplacian, so that for any $\varphi\in C^\infty(\mathbb{S}^1)$
\begin{align}
        (-\Delta)^s \varphi(x)\int_{\mathbb{S}^1}(\varphi(x)-\varphi(y))K^s(x-y)dy.
    \end{align}
For a description of $K^s$, see \cite{Roncal-Stinga}. Following \cite{Mazowiecka-Schikorra}, for any $F:\mathbb{S}^1\times\mathbb{S}^1\to\mathbb{R}$ we define the \textbf{$s$-fractional divergence of $F$} to be the distribution
\begin{align}
        \operatorname{div}_sF[\varphi]=\int_{\mathbb{S}^1}\int_{\mathbb{S}^1}F(x,y)(\varphi(x)-\varphi(y))K^{\frac{s}{2}}(x-y)dxdy.
    \end{align}
whenever the integral is well defined.
\begin{lem}\label{lem: link-fractional-divergence}
    Let $s\in (0,\frac{1}{2})$. Let $a, b, \varphi\in C^\infty(\mathbb{S}^1)$. Then
    \begin{align}
        \int_{\mathbb{S}^1}D_s(a,b)\varphi=2\operatorname{div}_{2s}(a(x)b(y))[\varphi].
    \end{align}
\end{lem}
\begin{proof}
We compute
    \begin{align}
        \int_{\mathbb{S}^1}D_s(a,b)\varphi=&\int_{\mathbb{S}^1}\int_{\mathbb{S}^1}((a(x)-a(y))b(x)-a(x)(b(x)-b(y)))\varphi(x)K^s(x-y)dxdy\\
        \nonumber
        =&\frac{1}{2}\int_{\mathbb{S}^1}\int_{\mathbb{S}^1}((a(x)-a(y))b(x)-a(x)(b(x)-b(y)))\varphi(x)K^(x-y)dxdy\\
        \nonumber
        &+\frac{1}{2}\int_{\mathbb{S}^1}\int_{\mathbb{S}^1}((a(y)-a(x))b(y)-a(y)(b(y)-b(x)))\varphi(y)K^s(x-y)dxdy\\
        \nonumber
        =&\int_{\mathbb{S}^1}\int_{\mathbb{S}^1}(a(x)b(y)-a(y)b(x))(\varphi(x)-\varphi(y))K^s(x-y)dxdy\\
        \nonumber
        =&2\int_{\mathbb{S}^1}\int_{\mathbb{S}^1}a(x)b(y)(\varphi(x)-\varphi(y))K^s(x-y)dxdy\\
        \nonumber
        =&2\operatorname{div}_{2s}(a(x)b(y))[\varphi],
    \end{align}
    where in the second step we interchanged the variables $x$ and $y$.
\end{proof}

\printbibliography

\end{document}